\DeclareSymbolFont{AMSb}{U}{msb}{m}{n}
\documentclass[11pt,a4paper,twoside,leqno,noamsfonts]{amsart}
\usepackage[utopia]{mathdesign}
\usepackage[english]{babel}
\usepackage{amsmath, amsthm, amstext}
\usepackage{microtype}
\usepackage{url}
\usepackage[usenames]{color}
\usepackage[all]{xy}
\usepackage{enumerate}
\usepackage{comment}
\xyoption{all}
\usepackage{srcltx} %per fer cerca inversa
\newtheorem{thm}{Theorem}[section]
\newtheorem{lem}[thm]{Lemma}
\newtheorem{cor}[thm]{Corollary}
\newtheorem{prop}[thm]{Proposition}

\theoremstyle{definition}
\newtheorem{rem}[thm]{Remark}
\newtheorem{defn}[thm]{Definition}

\newcommand{\bC}{\mathbb{C}}

\newcommand{\bG}{\mathbb{G}}

\newcommand{\bP}{\mathbb{P}}

\newcommand{\bZ}{\mathbb{Z}}

\newcommand{\cF}{\mathcal{F}}
\newcommand{\cG}{\mathcal{G}}
\newcommand{\cH}{\mathcal{H}}
\newcommand{\cM}{\mathcal{M}}
\newcommand{\cO}{\mathcal{O}}

\newcommand{\cU}{\mathcal{U}}
\newcommand{\cV}{\mathcal{V}}
\newcommand{\cS}{\mathcal{S}}

\newcommand{\hh}{\overline{\mathcal{H}}^{\mathrm{odd}}}
\newcommand{\mm}{\overline{\mathcal{M}}}

\newcommand{\gG}{G^1_{4, \mathrm{tr}}}
\newcommand{\gE}{G^1_{5,\mathrm{tr}}}

\newcommand{\hu}{\overline{\mathcal{H}}}

\newcommand{\lra}{\longrightarrow}

\DeclareMathOperator{\im}{Im}
\DeclareMathOperator{\Pic}{Pic}
\DeclareMathOperator{\Div}{div}

\begin{document}

\title[Alternating Catalan numbers]{Alternating Catalan numbers and curves with triple ramification}
%\author{G. Farkas, R. Moschetti, J.C. Naranjo, G.P. Pirola}

\author[G. Farkas]{Gavril Farkas}
\address{Gavril Farkas: Institut f\"ur Mathematik, Humboldt Universit\"at zu Berlin \hfill \newline\texttt{}
\indent Unter den Linden 6,
10099 Berlin, Germany}
\email{{\tt farkas@math.hu-berlin.de}}

\author[R. Moschetti]{Riccardo Moschetti}
\address{Riccardo Moschetti: Dipartimento di Matematica,  Universit`a degli Studi di Pavia \hfill \newline\texttt{}
\indent Via Ferrata, 1, 27100 Pavia, Italy}
\email{{\tt rmoschetti@gmail.com}}

\author[J.C. Naranjo]{Juan Carlos Naranjo}
\address{Juan Carlos Naranjo: Universitat de Barcelona, Departament de Matem`atiquesi Inform`atica\hfill
\hfill \newline\texttt{}
\indent Gran Via 585, 08007 Barcelona, Spain} \email{{\tt jcnaranjo@ub.edu}}

\author[G.P. Pirola]{Gian Pietro Pirola}
\address{Gian Pietro Pirola: Dipartimento di Matematica,  Universit`a degli Studi di Pavia \hfill
\indent \newline\texttt{}
\indent Via Ferrata, 1, 27100 Pavia, Italy}
 \email{{\tt gianpietro.pirola@unipv.it}}

\begin{abstract}
We determine the number of minimal degree covers of odd ramification for a general curve.
\end{abstract}

\maketitle

\section{Introduction}
The Catalan numbers $C_n:=\frac{1}{n+1}{2n \choose n}$ form one of the most ubiquitous sequence in classical combinatorics. Stanley's book \cite{St} lists $66$ different manifestations of these numbers in various counting problems. In the theory of algebraic curves, the Catalan number
$C_n$ counts the covers $C\rightarrow \mathbb P^1$ of \emph{minimal} degree $n+1$ from a general curve $C$ of genus $2n$. Each such cover has simple ramification and its monodromy group equals $S_{n+1}$. By degenerating $C$ to a rational $g$-nodal curve, it was already known to Castelnuovo \cite{C89} that the number of such covers coincides with the degree of the Grassmannian $G(2,n+2)$ in its Pl\"ucker embedding, which is well-known to equal $C_n$.

\vskip 3pt

It has been shown by Guralnick and Magaard \cite{GM} (see also \cite{GS}) that for a general curve $C$ of genus $g>3$, the monodromy group $M_f$ of each cover $f:C\rightarrow \bP^1$ is either the symmetric or the alternating group. For $g\leq 3$ several other groups do occur.
The aim of this paper is to determine the number of covers $f:C\rightarrow \mathbb P^1$ with alternating monodromy having as source a general curve $C$ of genus $g$ and such that $\mbox{deg}(f)$ is minimal (among covers with this property). The most natural case is when the local monodromy   around each branch point is given by a $3$-cycle. We refer to $f$ as being an \emph{odd cover}. A moduli count indicates that $f$ has $3g$ branch points and that $\mbox{deg}(f)=2g+1$.    Writing $D=2(x_1+\cdots+x_{3g})$ for the ramification divisor of $f$, from the Hurwitz formula it follows that $\vartheta:=\mathcal{O}_C(D)\otimes f^*\mathcal{O}_{\mathbb P^1}(-1)$ is a theta characteristic on $C$.  These coverings  and their relation with spin structures have already been studied in \cite{S}, \cite{S1} and \cite{F}. We denote by $\mathcal H_g^{\text{odd}}$ the Hurwitz space parametrizing odd covers $f:C\rightarrow \mathbb P^1$ of degree $2g+1$ with local monodromy at each branch point being given by a $3$-cycle. Fried showed in \cite{F} that $\mathcal{H}_g^{\text{odd}}$ has two connected components depending on the parity of $\vartheta$. The forgetful map
$$\varphi: \cH_g^{\text{odd}} \lra \cM_g , \    \mbox{ } \mbox{ }   \  \varphi\bigl([C\to \bP^1]\bigr):=[C]$$ is a map between varieties of  the same dimension $3g-3$.
Using an inductive argument, it is shown in \cite{MV}  that $\varphi$ exists, hence $\varphi$ is generically finite. Our aim is to determine its degree $\mathfrak{A}_g:=\mbox{deg}(\varphi)$. By analogy with the case of the symmetric group, we refer to $\mathfrak{A}_g$ as the $g$th \emph{alternating Catalan number}.

\vskip 3pt

\begin{thm}\label{main}
The number of odd covers of degree $2g+1$ of a general curve of genus $g\ge 3$ equals
$$\mathfrak{A}_g= 16^g \sum_{i=0}^g (-2)^i {g\choose i} C_{2g-i}.$$
\end{thm}

Unlike the classical Catalan numbers, their alternating counterparts $\mathfrak{A}_g$ do not admit a  closed formula. Instead, we determine their generating series.

\begin{thm}\label{maingen}
The generating series of the alternating Catalan numbers is the algebraic function
$$\sum_{g\geq 0} \mathfrak{A}_g w^{2g+1}=\frac{2w}{\sqrt{1+64w^2+16w\sqrt{16w^2+1}}+\sqrt{1+64w^2-16w\sqrt{16w^2+1}}}.$$
\end{thm}

An elementary calculation shows that the convergence radius of the algebraic function appearing in Theorem \ref{maingen} equals $\frac{\sqrt{2}}{16}$. Using \cite{FS} Theorem IV.7, we can thus determine the exponential growth rate of the numbers $\mathfrak{A}_g$ and we have
$$\mathfrak{A}_g=\Bigl(\frac{16}{\sqrt{2}}\Bigr)^{2g+1}\chi(g), \ \ \ \mbox{ where } \ \ \ \limsup_{g\rightarrow \infty}\sqrt[2g+1]{\chi(g)}=1.$$

\vskip 3pt
 
The proof of Theorem \ref{maingen} relies on applying the Lagrange Inversion Theorem to the series of expressions computed in Theorem \ref{main}. Theorem \ref{main} is proved by degenerating  a general curve of genus $g$ to a flag curve $C$ consisting of a smooth rational spine having $g$ elliptic tails attached to general points of the spine. One can explicitly exhibit all odd admissible covers of degree $2g+1$ having a source stably equivalent to $C$.  This is carried out in Section $3$. The formula appearing in Theorem \ref{main} depends on two initial values $N_4$ and $N_5$ which are related to the existence of certain odd maps of degrees $4$ and $5$ respectively on a general pointed elliptic curve $[E,P]\in \cM_{1, 1}$. Precisely $N_5$ counts the  covers
$E\stackrel{5:1}\rightarrow \bP^1$, which are totally ramified at $P$ and have three further triple ramification points. Similarly, $N_4$ is the number of degree $4$ covers $E \stackrel{4:1}\rightarrow \mathbb P^1$ having triple ramification at $P$ and at three further unassigned points. A significant part of the paper is devoted to proving that $N_4=N_5=16$, see Theorems \ref{main_section4} and \ref{main_section5}. We present two independent proofs of this fact. The first, uses the theory of elliptic functions. Inspired by  a method from  \cite{AP} we study the existence of such odd maps by counting the solutions of a certain differential equation on an elliptic curve. A Chern class calculation shows that $N_4,N_5 \le 16$. Then we show for a particular elliptic curve that the number of solutions is exactly $16$. The second proof of the equality $N_4=N_5=16$, is carried out in Section \ref{sec:deg}, see Theorems \ref{N5:attempt2} and \ref{N4:attempt2} respectively. It relies on degeneration to a nodal elliptic curve and involves rather subtle intersection-theoretic calculations on moduli stacks of odd admissible covers.

\vskip 3pt

Soon after the appearance of this paper, Lian's related work \cite{Li} was posted on \emph{arXiv}. He considers more general enumerative problems for pencils on curves than we do, though in the specific situation described in this paper his results are less explicit.

\vskip 3pt

\noindent {\bf Acknowledgments:} We are grateful to both M. Fried and D. Oprea for very useful discussions related to this circle of ideas.

\vskip 3pt

{\small{Moschetti is member of GNSAGA (INDAM) and is partially supported by MIUR: Dipartimenti di Eccellenza Program (2018-2022)-Dept. of Math. Univ. of Pavia; Naranjo was partially supported by the Proyecto de Investigaci\'on MTM2015-65361-P; Pirola is member of GNSAGA (INDAM) and is partially supported by PRIN Project \emph{Moduli spaces and Lie theory} (2017) and by MIUR: Dipartimenti di Eccellenza Program (2018-2022) - Dept. of Math. Univ. of Pavia. Farkas was supported by the DFG Grant \emph{Syzygien und Moduli}}}.

\section{Preliminaries}
We collect a few things that will be used throughout the paper.

\subsection{Monodromy of coverings and Hurwitz spaces of odd covers}

Let $f:C\rightarrow \mathbb P^1$ be a finite cover of degree $d$ and denote by $B:=\{P_1, \ldots, P_n\}$ its branch locus. For a point $Q\in \mathbb P^1\setminus B$, let $$\rho_f:\pi_1(\mathbb P^1\setminus B, Q)\rightarrow S_d$$ be its monodromy representation. We denote by $M_f:=\mbox{Im}(\rho_f)$ the monodromy group of $f$. The local monodromy of $f$ around a branch point $P_i\in B$ is given by $\tau_i:=\rho_f([\gamma_i])\in S_d$, where $\gamma_i$ is a simple loop around $P_i$ based at $Q$. The cover $f$ is said to be \emph{alternating} if $M_f\subseteq A_d$. We shall often consider alternating covers $f:C\rightarrow \mathbb P^1$, such that each local monodromy  $\tau_i$ is given by an odd cycle. We refer to such an $f$ as being an \emph{odd cover}.

\vskip 3pt

We denote by $\mathcal{H}_g^{\mathrm{odd}}$ the Hurwitz space parametrizing odd covers $f:C\rightarrow \mathbb P^1$ of degree $2g+1$ branched at $3g$ points. We require that the  local monodromy around each branch point of $f$ be given by a $3$-cycle. Such a cover is endowed with a theta characteristics  $\vartheta:=\mathcal{O}_C(D)\otimes f^*(\mathcal O_{\mathbb P^1}(-1))$, where $D$ is the half of the ramification divisor $R_f$.
As proved by Mumford (see \cite {Mu}), the parity of the spin structure $\vartheta$ is a deformation invariant. Two odd covers $f_1:C_1\rightarrow \bP^1$ and $f_2:C_2\rightarrow \bP^1$ are identified as points in $\mathcal{H}_g^{\mathrm{ord}}$ when there exists an isomorphism $\tau:C_1\rightarrow C_2$ and an automorphism $\bar{\tau}:\bP^1\rightarrow \bP^1$ such that $f_2\circ \tau=\bar{\tau}\circ f_1$. We denote by $[f:C\rightarrow \bP^1]\in \mathcal{H}_g^{\mathrm{ord}}$ the moduli point of the cover $f$.

\vskip 3pt

Let $\hh_g$ be the compactification of $\mathcal{H}_g^{\mathrm{odd}}$ by admissible $A_{2g+1}$-covers. By \cite{ACV}, the stack $\hh_g$ is isomorphic to the stack of \emph{balanced
 twisted stable} maps into the classifying stack $\mathcal{B} A_{2g+1}$,  that is,
$$\hh_g:=\overline{\mathcal{M}}_{0,3g}\Bigl(\mathcal{B} A_{2g+1}\Bigr)/S_{3g},$$
where the action of the symmetric group $S_{3g}$ is given by permuting the branch points. For details concerning the construction of the space of admissible covers we refer to \cite{ACV}. Note
that $\hh_g$ is the normalization of the space $\mathcal{HM}_g^{\mathrm{odd}}$ of  Harris-Mumford admissible covers introduced in \cite{HM}.

Points of $\hh_g$ are odd admissible coverings
$[f:X\rightarrow \Gamma, P_1+\cdots+P_{3g}]$, where $X$ and $\Gamma$ are nodal curves of genus $g$ and $0$ respectively, $f$ is a finite map of degree $2g+1$ with $f^{-1}(\Gamma_{\mathrm{sing}})=X_{\mathrm{sing}}$ and $P_1, \ldots, P_{3g}\in  \Gamma_{\mathrm{reg}}$
are the branch points of $f$. The local monodromy of $f$ around $P_i\in \Gamma$ is given by a $3$-cycle  $\tau_i\in A_{2g+1}$, for $i=1, \ldots, 3g$.
The local monodromy of $f$ at both branches of $X$ at a node $p\in X_{\mathrm{sing}}$ is given by an alternate permutation, which is not necessarily a $3$-cycle. We denote by $$\varphi:\hh_g\rightarrow \mm_g$$ the map associating to an admissible cover $[f:X\rightarrow \Gamma, P_1+\cdots+P_{3g}]$ the stable model $\mbox{st}(X)$ of its source. As discussed in the Introduction, $\varphi$ is a generically finite map.

\vskip 4pt

We discuss the local structure of the space of admissible covers following \cite{HM} p.62.  We fix a point $\xi:=[f:X\rightarrow \Gamma, \ P_1+\cdots+P_{3g}]$ as above and assume $\Gamma_{\mathrm{sing}}=\{u_1, \ldots, u_r\}$. For $i=1, \ldots, r$, set $f^{-1}(u_i)=\{Q_{i,1},\ldots, Q_{i,\ell_i}\}\subseteq X_{\mathrm{sing}}$. The (non-normalized) space
$\mathcal{HM}_{g}^{\mathrm{odd}}                                                                                                                  $ is described by its local ring
\begin{equation}\label{localring}
\hat{\cO}_{\xi, \mathcal{HM}_g^{\mathrm{odd}}}=\mathbb C
\bigl[\bigl[t_1, \ldots, t_{3g-3}, s_{i,1}, \ldots, s_{i, \ell_i}, \ \ i=1, \ldots, r\bigr]\bigr]/s_{i,1}^{\mu_{i,1}}=\cdots=s_{i,\ell_i}^{\mu_{i,\ell_i}}=t_i, \ \ i=1, \ldots, r,
\end{equation}
where $t_i$ is the local corresponding to smoothing the node $u_i\in \Gamma$ and $(\mu_{i,1}, \ldots, \mu_{i,\ell_i})$ describes the ramification profile of $f^{-1}(u_i)$. In particular, $\mathcal{HM}_g^{\mathrm{odd}}$ (and hence $\hh_g$)
is smooth at $t$ whenever over each node $u_i$ with $i=1, \ldots, r$ there exist at most one ramification point, that is, at most one index $j\in \{1, \ldots, \ell_i\}$ with $\mu_{i,j}>1$.

\subsection{Schubert cycles with respect to osculating flags to rational normal curves} We recall the definition of Schubert cycles in the Grassmannian of lines $\mathbb G:=G(2,V)$, where $V\cong \mathbb C^n$. After choosing a flag
$F_{\bullet}: V=V_n\supset V_{n-1}\supset \ldots \supset V_0=0$, for a decreasing sequence of positive integers $\mu:=(\alpha_1\geq \alpha_0)$ we introduce the Schubert cycle
$$\sigma_{\mu}=\sigma_{\mu}(F_{\bullet}):=\bigl\{\Lambda\in \mathbb G: \Lambda\subseteq V_{n-\alpha_0}, \ \Lambda\cap V_{n-\alpha_1-1}\neq 0\bigr\}.$$
When the meaning of the flag $F_{\bullet}$ is clear from the context, we shall drop it from the notation of the corresponding Schubert cycle.
Note that $\mbox{codim}(\sigma_{\mu}, \mathbb G)=|\mu|=\alpha_0+\alpha_1$.

\vskip 3pt

When counting admissible covers we often use non-generic flags defined in terms of a rational normal curve $R\subseteq \mathbb P^{n-1}$ embedded by $V:=H^0(\mathbb P^1, \mathcal{O}_{\mathbb P^1}(n-1))$. For a point $P\in R$, let $F_{\bullet}(P)$ be the osculating flag of $R$ at $P$, thus $V_i:=H^0\Bigl(\mathbb P^1, \cO_{\bP^1}(n-1)\bigl(-(n-i)P\bigr)\Bigr)$ for $i=0, \ldots, n-1$. The osculating flags to $R$ enjoy two very desirable transversality properties:
\begin{enumerate}
\item For any number of distinct points $P_1, \ldots, P_s\in R$ and any partitions $\mu_1, \ldots, \mu_s$, the intersection $\cap_{i=1}^s \sigma_{\mu_i}\bigl(F_{\bullet}(P_i)\bigr)$ has the expected dimension $2(n-2)-|\mu_1|-\cdots-|\mu_s|$, see \cite{EH1}, Theorem 2.3.
\item If all the points $P_i$ are in $\bP^1(\mathbb R)$ and $|\mu_1|+\cdots+|\mu_s|=2(n-2)$, then the intersection
$\cap_{i=1}^s \sigma_{\mu_i}\bigl(F_{\bullet}(P_i)\bigr)$ is a \emph{reduced} union of real points, see \cite{MTV}.
\end{enumerate}

\vskip 4pt

\subsection{Odd covers and differential equations on elliptic curves.}\label{subsectdiff}

Let $E$ be a complex elliptic curve and fix a point $P\in E$. We consider the group structure on $E$ having the point $P\in E$ as origin. We have $E\cong \bC/\Lambda$, where $\Lambda$ is a lattice generated by $1$ and $\tau$, where  $\im(\tau)>0$. Let $\pi:\bC\to E$ be the universal covering, so that $\pi(0)=P$ and we denote by $\sigma:E\to E$ the involution fixing  $P$, which can be thought as the involution associated to the hyperelliptic linear series $|2P|$, which induces a map $h:E\to \bP^1$. We write  $D:=P+Q+R+S$ for the ramification divisor of $h$, thus $Q, R$ and $S$  are the point of order two on $E$.
The function $h$ is determined explicitly by the Weierstrass function $\wp$ (see \cite{AMS} and \cite{Lang}), which is given by
$$\wp(z)=\frac{1}{z^2}+\frac{1}{20}g_2 z^2+\frac{1}{28}g_3 z^4+ O(z^6),$$
where $g_2$ and $g_3$ depend on the choice of the lattice $\Lambda$. Consider the image of the half period
$$e_1:=\wp\bigl(\frac{1}{2}\bigr),\ \ e_2:=\wp\bigl(\frac{\tau}{2}\bigr),\ \ e_3:=\wp\bigl(1+\frac{\tau}{2}\bigr).$$

We record  the following relations between $\wp$ and its derivatives
\begin{align}
\wp'(z)^2&=4\wp(z)^3-g_2\wp(z)-g_3=4(\wp(z)-e_1)(\wp(z)-e_2)(\wp(z)-e_3), \label{eqn:propP1}\\
\wp''(z)&=6\wp(z)^2-\frac{1}{2}g_2. \label{eqn:propP2}
\end{align}

The Weierstrass form of $E$ is given by $y^2=4x^3-g_2x-g_3=4(x-e_1)(x-e_2)(x-e_3)$
and the $j$-invariant of $E$ is computed by the well-known formula
$$j_E= 1728 \frac{g_2^3}{g_2^3-27g_3^2}.$$
The field of the rational function $\bC(E)$ is isomorphic to the subfield $\bC(\wp,\wp')$ of the complex meromorphic functions generated by $\wp$ and $\wp'$, see \cite{Lang}.

%Since the intersection point of the elliptic curve $E$ with the rational curve in the degeneration procedure has to be one of the ramification points we can consider that there is a marked point $P$ which will be the the origin of the curve as a group. We keep the notation introduced in section $2$, hence $Q, R, S$ are the points satisfying $2P\sim 2Q\sim 2R \sim 2S$ and $P-Q, P-R, P-S$ are the non-trivial $2$-torsion points on $E$. Also we will use the map $h$ attached to $|2P|$.

\vskip 4pt

As described in the Introduction, an odd map $f:E \to \bP^1$ comes equipped with a spin structure $\vartheta=\cO_E(D)\otimes f^*\cO_{\bP^1}(-1)\in \Pic^0(E)$, where $D:=\frac{1}{2} R_f$ is half of the ramification divisor $R_f$ of $f$. Hence there are four possibilities, meaning
$$\vartheta=\cO_E,\ \ \vartheta=\cO_E(P-Q),\ \ \vartheta=\cO_E(P-R),\ \ \vartheta=\cO_E(P-S).$$

We fix a non-trivial holomorphic form $dz$ in the space of holomorphic differentials $H^0(E,\omega_E)$. We may assume that $h$, viewed as a meromorphic function, has a second order pole at $P$ and a second order zero at $Q$. The next proposition allows us to translate the computation of the quantities $N_4$ and $N_5$, essential in proving Theorem \ref{main} into finding the solutions of certain differential equations on $E$.

\begin{prop}\label{equation}
Let $\vartheta \in \mathrm{Pic}^0(E)[2]$. A meromorphic function $f$ corresponds to an odd cover $f:E\rightarrow \bP^1$ with associated spin structure $\vartheta$ if and only if there exists a meromorphic function $s$ on $E$ with
\begin{equation} \label{main_equation}
 df =s^2\omega,
\end{equation}
where $\omega =dz$ if $\vartheta\cong \cO_E$, and $\omega =hdz$ if $\vartheta\cong \cO_E(P-Q)$ respectively.
\end{prop}
\begin{proof}
Assume $f:E\rightarrow \bP^1$ is an odd function with trivial spin structure $\vartheta\cong \cO_E$.  Let $A:=f^*(\infty)$ be the divisor of poles of $f$. Since $D-A$ is a principal divisor, there exists a meromorphic function $s\in \mathbb C(E)$  with
$$\Div(s)=D-A.$$
Then, $\Div(s^2)=2D-2A=R_f-2A$, which is precisely the divisor of $d f$. Up to modifying $s$ by a constant, the equation $df =s^2dz$ is satisfied. In the opposite direction, if $f$ satisfies equation (\ref{main_equation}), then by taking local coordinates it is clear that $f$ is odd.

\vskip 3pt

The case when the associated spin structure $\vartheta$ is even is similar. With the same notation we have that $D-A$ is linearly equivalent to $Q-P$, hence there exists $s\in \mathbb C(E)$ with
$$\Div(s)=D-A+P-Q.$$
Therefore, $\Div(s^2h)=R_f-2A$, since $\Div(h)=2Q-2P$. As above, after rescaling we may assume that $df =s^2hdz$. Assuming conversely that $f$ satisfies this equation, a simple local analysis shows that $f$ is an odd function.
\end{proof}

\begin{rem}
Proposition \ref{equation} is valid for odd covers of arbitrary genus. Consider a curve $C$ of genus $g$ and a theta characteristic $\vartheta$ on $C$. Then a cover $f:C\to \bP^1$ of degree $2g+1$ is odd with associated spin structure  $\vartheta$ if and only if there exists a divisor $A$ of degree $2g+1$ such that $df = s^2\in H^0\bigl(C, \omega_C(2A)\bigr)$, for $s\in H^0(C,\vartheta(A))$.

%The idea, roughly speaking, goes as follows: if $s$ exists, looking at the local expressions, one easily deduces that $\gamma $ is odd. On the other hand, assume that $\gamma$ exists and put $A$ to the divisor of its poles. Let $R=2 R_{\gamma }$ be the ramification divisor. Then $\Div(d\gamma )= 2R_{\gamma }-2A$. On the other hand $L=\cO_C(R_{\gamma }-A)$ by definition. Then choose a section $s_0\in H^0(C, L(A))$ and put $\Div(s_0)=Z-A$. Hence $Z\sim R_{\gamma}$ and there exists a meromorphic function $f$ such that $\Div(f)=Z-R_{\gamma }$. One easily checks that $s=f s_0$ satisfies the equation.
\end{rem}

In Section \ref{diffeq2} we shall study the solutions of the equation (\ref{main_equation}) when $\deg (f)=4$ and $f$ has a triple ramification point at $P$ and when $\deg (f)=5$ and $P$ is a point of total ramification of $f$ respectively.

\vskip 3pt

\section{Odd admissible covers on flag curves of genus $g$}
In this section we apply degeneration methods in order to prove the formula (\ref{degeneration_formula}) below. This is an intermediate step in the proof of the main Theorem (\ref{main}). We recall that for a pencil $\ell\in G^1_d(C)$ on a smooth curve $C$, for a point $P\in C$ we denote by $a^{\ell}(P)=\bigl(a_0^{\ell}(P)<a_1^{\ell}(P)\bigr)$ its vanishing sequence at $P$ and by $\alpha^{\ell}(P)=(\alpha^{\ell}_0(P)=a_0^{\ell}(P), \alpha^{\ell}_1(P)=a^{\ell}_1(P)-1)$ its ramification sequence. We fix a general pointed elliptic curve $[E,P]\in \cM_{1,1}$ and we introduce two loci. Firstly,

$$\gG(E,P):=\Bigl\{\ell\in G^1_4(E):a_1^{\ell}(P)\geq 3, \mbox{ there exist distinct points } P_i \in E\setminus \{P\} \mbox{ with } a_1^{\ell}(P_i)\geq 3, \ i=1,2,3 \Bigr\}.$$
A pencil $\ell\in \gG(E,p)$ corresponds to a cover $f:E\rightarrow \bP^1$ of degree $4$ ramified triply at $P, P_1, P_2, P_3$ and having no further ramification points. Secondly, we define the locus
$$\gE(E,P):=\Bigl\{\ell\in G^1_5(E):a_1^{\ell}(P)\geq 5, \mbox{ there exist distinct points } P_i \in E\setminus \{P\} \mbox{ with } a_1^{\ell}(P_i)\geq 3, \ i=1,2,3 \Bigr\}.$$
A pencil $\ell\in \gE(E,P)$ corresponds to a cover $f:E\rightarrow \bP^1$ totally ramified at $P$, triply ramified at $P_i$ for $i=1,2,3$ and having no further ramification points.

\vskip 3pt

A parameter count yields that both $\gG(E,P)$ and $\gE(E,P)$ are $0$-dimensional and we denote
$$N_4:=|\gG(E,P)| \ \ \mbox{ and } \  \ N_5:= |\gE(E,P)|$$ respectively, their cardinalities. We shall later prove that both $\gG(E,P)$ and $\gE(E,P)$ are reduced, but for now we do not need that.

\vskip 4pt

We fix once and for all a flag curve
$$[C:=R\cup_{Q_1} E_1 \cup \ldots \cup_{Q_g} E_g]\in \mm_g$$ consisting of a smooth rational curve $R$ and $g$ elliptic tails $E_i$ meeting the spine $R$ at the point $Q_j$ for $j=1, \ldots, g$. We require that $[E_j,Q_j]\in \cM_{1,1}$ are general, which in practice means that $E_j$ is not isomorphic to the Fermat cubic. The use of such flag curves in proving the classical Brill-Noether Theorem is well documented, see \cite{EH}.

\begin{thm}\label{hurwitzfibre}
The fibre of the morphism $\varphi:\hh_g\rightarrow \mm_g$ over the point $[C]$ can be described as
$$\varphi^{-1}\bigl([C]\bigr)=\bigcup_{J\subseteq \{1, \ldots,g\}}\Bigl( \prod_{j\in J}\gG(E_j,Q_j)\times \prod_{i\in J^c}\gE(E_i,Q_i)\Bigr) \times
\Bigl(\bigcap_{j\in J} \sigma_{3,1}\bigl(F_{\bullet}(Q_j)\bigr)\cap \bigcap_{i\in J^c} \sigma_{4,0}\bigl(F_{\bullet}(Q_i)\bigr)\Bigr).$$
Futhermore, if the points $Q_i\in R$ are chosen generically, the above cycle is $0$-dimensional and reduced.
\end{thm}

An immediate consequence of Theorem \ref{hurwitzfibre} is the following formula for the alternating Catalan number $\mathfrak{A}_g:=\mbox{deg}(\varphi)$.

\begin{thm} \label{degeneration_formula}
The number of odd coverings of degree $2g+1$ in a generic curve of genus $g\ge 3$ is
$$\mathfrak{A}_g=\Bigl(N_4\ \sigma_{4,0}+N_5\ \sigma_{3,1}\Bigr)^g\in H^{\mathrm{top}}\bigl(G(2,2g+2), \mathbb Z\bigr).$$
\end{thm}
\begin{proof}
In the statement of Theorem \ref{hurwitzfibre}, we sum over $k:=|J|\leq g$ to obtain
$$\mathfrak{A}_g=\sum_{k=0}^g {g\choose k} N_4^{k} N_5^{g-k} \sigma_{3,1}^k\cdot \sigma_{4,0}^{g-k}=\Bigl(N_4\ \sigma_{4,0}+N_5\  \sigma_{3,1}\Bigr)^g \in H^{\mathrm{top}}\bigl(G(2,2g+2), \mathbb Z\bigr).$$
\end{proof}

\noindent \emph{Proof of Theorem \ref{hurwitzfibre}.} We start with an odd admissible cover $\bigl[f:X\rightarrow \Gamma, P_1+\cdots+P_{3g}\bigr]$ of degree $2g+1$, having as source a nodal curve $X$ stably equivalent to $C$. For $i=1, \ldots, 3g$, we denote by $x_i\in f^{-1}(P_i)$ the unique odd ramification point lying over the branch point $P_i$. We fix an index $j\in \{1, \ldots, g\}$ and consider the restriction $f_j=f_{|E_j}:E_j\rightarrow \bP^1$, where $\bP^1$ is one of the components of $\Gamma$. Let $a=a^{f_j}_1(Q_j)$ be the vanishing index of the point of attachment $Q_j$ and note that away from $Q_j$, the ramification points of $f_j$ are precisely those points $x_i$ with $i=1, \ldots, 3g$ which lie on $E_j$. Set $d_j:=\mbox{deg}(f_j)$.

\vskip 3pt

We claim that at least three of the odd ramification points of $f$ lie on $E_j$. Indeed, assume first that, on the contrary, at most one such point lies on $E_j$.
By the Hurwitz formula, then $2d_j=\mbox{deg}(R_{f_j})=a-1+2=a+1\leq d_j+1$, hence $d_j=1$, which is impossible. If two odd ramification points lie on $E_j\setminus\{Q_j\}$, then by the same reasoning $2d_j=a-1+2+2\leq d_j+3$, hence $d_j=3$. But then $f_j:E_j\rightarrow \bP^1$ is a degree $3$ cover having three total ramification points, which forces $E_j$ to be isomorphic to the Fermat cubic, in particular to have $j$-invariant zero, contradicting the generality of $E_j$. Thus at least three of the points $x_1, \ldots, x_{3g}$ specialize on each tail $E_j$. Since there are precisely $g$ elliptic tails, this implies that
precisely three ramification points lie on each of $E_1, \ldots, E_g$, whereas the spine $R$ contains no ramification points.

\vskip 4pt

Applying once more the Hurwitz formula to $f_j$, we get $2d_j-2=a-1+6\leq d_j+5$. If $d_j=5$, then $a=5$ and the pencil corresponding to $f_j$ belongs to $\gE(E_j, Q_j)$, in particular there are $N_5$ choices for $f_j$. If, on the other hand, $d_j=4$, then $a=3$ and the pencil corresponding to $f_j$ gives rise to a point in $\gG(E_j, Q_j)$. Let $J\subseteq \{1, \ldots, g\}$ be the set of labels for the elliptic curves $E_j$ with $d_j=4$, in which case $J^c$ contains the labels for those elliptic tails $E_j$ having $d_j=5$. Set $|J|=:k\leq g$. For each $j\in J$, writing  $f_j^{-1}(f_j(Q_j))=\{Q_j, Q_j'\}$, it follows that there exists a rational component $R_j'$ of $X$ meeting $E_j$ at $Q_j'$ and such that $f(R_j')=f(R)$. In fact $\mbox{deg}(f_{|R_j'})=1$. It follows that the degree of the restriction $f_R=f_{|R}:R\rightarrow \bP^1$ is then at most $2g+1-|J|=2g+1-k$. Since the ramification indices at $Q_j$ on the two branches of $R$ and $E_j$ must agree for $j=1, \ldots, g$, it follows that $f_R$  corresponds to a pencil $\ell_R\in G^1_d(R)$ having vanishing $a_1^{\ell_R}(Q_i)\geq 5$, for $i\in J^c$ and $a_1^{\ell_R}(Q_j)\geq 3$ for $j\in J$. The Hurwitz formula applied to $f_R$ implies that  $d\geq 2g+1-k$, hence $d=2g+1-k$. Equivalently, the pencil $$\ell:=\ell_R\Bigl(\sum_{j\in J} Q_j\Bigr)\in G^1_{2g+1}(R)$$ obtained from $\ell_R$ by adding base points at all points with labels from $J$ satisfies $a^{\ell}(Q_j)\geq (1,4)$ for $j\in J$ and $a^{\ell}(Q_i)\geq (0,5)$ for all $i\in J^c$.

\vskip 3pt

Write $\ell= \bigl(\cO_{\bP^1}(2g+1), V\bigr)$, for a subspace of sections  $V\in G\bigl(2, H^0(\bP^1, \cO_{\bP^1}(2g+1))\bigr)=G(2,2g+2)$. We regard $R$ as a rational normal curve in $\bP^{2g+1}$ embedded each point by $H^0(R, \cO_R(2g+1))$ and denote by $F_{\bullet}(Q_j)$ the osculating flag at $Q_j$. Then for each vanshing sequence $(a_0<a_1)$ the condition $a^{\ell}(Q_j)\geq (a_0,a_1)$ is equivalent to $V\in \sigma_{a_1-1,a_0}\bigl(F_{\bullet}(Q_j))$, which establishes the Theorem set-theoretically.

\vskip 3pt

Observe that all covers $[f]\in \varphi^{-1}\bigl([C]\bigr)$ correspond to smooth points of $\hh_g$. Indeed the rational target curve $\Gamma=f(X)$ has $g+1$ components, namely $f(R)$ and $f(E_j)$ where $j=1,\ldots, g$. Over each node $f(Q_j)\in \Gamma_{\mathrm{sing}}$ lies
a \emph{single} ramification point, which using the local description (\ref{localring}) implies that $\hh_g$ is smooth at $[f]$. Furthermore, the fibre $\varphi^{-1}\bigl([C]\bigr)$ is \emph{scheme-theoretically} isomorphic to disjoint unions of copies of the intersection of Schubert cycles
$$\Bigl(\bigcap_{j\in J} \sigma_{3,1}\bigl(F_{\bullet}(Q_j)\bigr)\cap \bigcap_{i\in J^c} \sigma_{4,0}\bigl(F_{\bullet}(Q_i)\bigr)\Bigr).$$
Following \cite{MTV} this intersection is transverse when the points $Q_j\in R$ are general, which finishes the proof.

\hfill $\Box$

\section{Counting odd covers of elliptic curves I: an approach via differential equations}\label{diffeq2}

The goal of this Section is to determine the quantities $N_4$ and $N_5$ appearing in Theorem \ref{degeneration_formula}.
Thanks to Proposition \ref{equation}, these two problems can be reformulated in terms of differential equations of type (\ref{main_equation}). The result will follow by combining the upper bound provided by Proposition \ref{prop:upperbound4} and the lower bound provided by Proposition \ref{prop:lowerbound4}.

\subsection{Odd covers of degree $4$ on an elliptic curve.} We will ultimately prove the following result:

\begin{thm} \label{main_section4}
The number $N_4$ of odd maps $f:E\rightarrow \bP^1$ of degree $4$ from a general pointed elliptic curve $[E,P]\in \cM_{1,1}$ is equal to $16$.
\end{thm}

We shall use the same terminology of Subsection \ref{subsectdiff}. The origin $P$ of the elliptic curve $E$ may be assumed to be one of the ramification points, hence the odd function $f$ we are looking for belongs to $H^0(E, \cO_E(3P+x))$, for some point $x \in E$.
Assume first  $x\neq P$. By considering the local expression of $f$ around $P$ and $x$ and taking derivatives, we obtain that $df$ has a pole of order $4$ at $P$ and a pole of order $2$ at $x$, that is, $df \in H^0\bigl(E, \cO_E(4P+2x)\bigr)$. Differentiation  provides a linear map
$$\delta_x: H^0\bigl(E, \cO_E(3P+x)\bigr) \to H^0\bigl(E, \omega_E(4P+2x)\bigr)\cong H^0\bigl(E, \cO_E(4P+2x)\bigr).$$
This analysis also works when $x=P$. Then $H^0\bigl(E, \cO_E(4P+2x)\bigr)=H^0\bigl(E, \cO_E(6P)\bigr)$; the function $f$ has a pole of order $4$ at $P$ and so $df$ has there a pole of order $5$ at $P$ and $df \in H^0\bigl(E,\cO_E(5P)\bigr)$, which lies inside $H^0\bigl(E, \cO_E(6P)\bigr)$.
\medskip

Consider now a meromorphic function $s$ satisfying equation (\ref{main_equation}). If the associated theta characteristic $\vartheta$ is trivial, we have $s\in H^0\bigl(E, \cO_E(2P+x)\bigr)$, and we can consider the (non-linear) map
\begin{equation*}
\begin{aligned}
 \alpha_{x}: H^0\bigl(E, \cO_E(2P+x)\bigr) & \to H^0\bigl(E, \cO_E(4P+2x)\bigr) \\
 s & \mapsto s^2 dz.
\end{aligned}
\end{equation*}
If $\vartheta \cong \cO_E(P-Q)$, we consider a similar map $\alpha_x: H^0\bigl(E, \cO_E(P+Q+x)\bigr)\rightarrow H^0\bigl(E, \cO_E(4P+2x)\bigr)$ defined by $\alpha_x(s):=s^2 hdz$. The following corollary relates the maps $\delta_x$ and $\alpha_x$.

\begin{cor} \label{cor:intersection}
The solutions of equation ($\ref{main_equation}$) coincides with the intersection of the images of the maps $\delta _x$ and $\alpha_{x}$, as $x\in E$ varies.
\end{cor}
\begin{proof}
This follows directly from the proof of Proposition \ref{equation}. The map $\delta_x$ correspond to the left hand side of Equation (\ref{main_equation}), while the map $\alpha_x$ correspond to the right hand side.
\end{proof}

In order to exploit this result, notice that the map $\delta_x$ is linear, so in order to intersect its image with the one of $\alpha_x$ it is convenient to look at the kernel of the following composition
$$H^0\bigl(E, \cO_E(2P+x)\bigr) \xrightarrow{\alpha_x} H^0\bigl(E, \cO_E(4P+2x)\bigr) \to \frac {H^0\bigl(E, \cO_E(4P+2x)\bigr)}{\delta_x\bigl(H^0(E, \cO_E(3P+x))\bigr)}.$$

We regard these maps globally by moving the point $x\in E$.
To that end, we consider the projections $\pi_i: E \times E \to E$ for $i=1,2$ and the diagonal $\Delta\subseteq E\times E$.
For a point $x\in E$, define $E_{x}$ as $E\times \{x\}:=\pi_2^*(x)$. For an effective divisor $A=\sum n_i P_i$ on $E$, we set $E_A:=\sum n_i E_{P_i}$.

\begin{defn}
For an integer $m$ and an effective divisor $A$ on $E$,  we define the vector bundle
$D_{m,A}:=\pi_{1*} \bigl(\cO_{E \times E}(m\Delta+ E_A)\bigr)$ on $E$.
\end{defn}

Observe that $\cG:=D_{1,3P}$, $\cF:=D_{1,2P}$ and $\cU:=D_{2,4P}$ are the vector bundles we mentioned before. We have the natural identifications for the respective fibres
$$\cG(x)\cong H^0\bigl(E, \cO_E(3P+x)\bigr), \ \cF(x)\cong H^0\bigl(E,\cO_E(2P+x)\bigr)\ \mbox{ and }\  \cU(x)\cong H^0\bigl(E,\cO_E(4P+2x)\bigr).$$

%For a point $x\in E$, the map $\alpha_x$ is induced by a map $\alpha:\cF\rightarrow \Gamma$ obtained by applying $\pi_{1 *}$ to the map of sheaves given by squaring sections
%$$\cO_{E \times E}(\Delta+ 2 E_{P}) \lra  \cO_{E \times E}(2\Delta+ 4 E_{P}).$$

The map $\delta_x$ correspond to a sheaf morphism $\delta:\cG\rightarrow \cU$: start by considering the differential
$$d: \cO_{E \times E} \to \Omega ^1_{E\times E} = \pi_1^* \cO_E \oplus \pi_2^* \cO_E \cong \cO_{E\times E}^{\oplus 2}.$$
From the effective divisor $A=\sum n_iP_i$, let us define the augmented divisor $A^a:=\sum (n_i+1)P_i$. The differential $d$ induces a map of sheaves on $E\times E$ by just taking derivatives of sections with poles along the divisors $\Delta$ and $E_{P_i}$:
$$\cO_{E\times E} (m \Delta +  E_A) \to \Omega ^1_{E\times E} ((m+1)\Delta + E_{A^a}) \cong (\cO_{E\times E} ((m+1)\Delta + E_{A^a}))^{\oplus 2}.$$
By projecting to the first summand and applying the functor $\pi_{1 *}$ we get the map of sheaves $\delta: \cG\to \cU$, which glues the maps $\delta_x$.

The upper bound given by Proposition \ref{prop:upperbound4} comes from the computation of the Chern classes of the sheaves involved in the picture above. Let us begin this final computation with the following:

\begin{lem} \label{lem:c1Dma}
For all $m\geq 0$ and for all effective divisors $A$,  we have that $c_1(D_{m,A}) = m\cdot \mathrm{deg}(A)$.
\end{lem}
\begin{proof}
We first check the formula for $n=0$, in which case $A=0$. Consider the short exact sequence
\begin{equation} \label{case_n=0}
0 \longrightarrow \cO_{E \times E}((m-1)\Delta) \longrightarrow \cO_{E \times E}(m\Delta) \longrightarrow \cO_{\Delta}(\Delta) \longrightarrow 0.
\end{equation}

By the adjunction formula $\cO_{\Delta}(\Delta)$ is trivial. If $m\ge 1$, then $R^1 \pi_{1 *} \cO_{E \times E}((m-1)\Delta)=0$ and we get immediately that $c_1(D_{m, 0})=c_1(D_{m-1,0})$. For $m=0$ we have
$$c_1(D_{0,0})=c_1(\pi_{1 *}   \cO_{E \times E})=c_1(\cO_E)=0.$$
Therefore, $c_1(D_{m,0})=0$ for all $m\ge 0$. \medskip

Assume now $n>0$. After tensoring the short exact sequence (\ref{case_n=0}) with $\cO_{E \times E}(E_A)$, we apply the functor $\pi_{1 *}$. Since $n>0$, we have
$R^1\pi_{1*} \cO_{E \times E}\bigl((m-1)\Delta +E_A\bigr)=0$
for any $m\ge 1$, therefore the following sequence is exact
$$ 0 \longrightarrow D_{m-1,n} \longrightarrow D_{m,n} \longrightarrow \pi_{1 *} \cO_{\Delta}(E_A) \longrightarrow 0.$$
The first Chern class of $\pi_{1 *}\cO_{\Delta}(E_A)$ equals $c_1(\cO_{E}(A))=\deg (A)=n$. Then $c_1(D_{m,n}) = n + c_1(D_{m-1,n})$ for $n>0$. We can repeat this procedure until we get
$c_1(D_{m,n}) = m\cdot n + c_1(D_{0,n})$. \medskip

It remains to take care of the case $m=0$ and $n>0$. Let $P$ be a point in the support of $A$ and set $A_0=A\setminus \{P\}$.
Similarly to the previous cases, one finds for $n\ge 2$ the following short exact sequence
$$0 \longrightarrow D_{0,A_0}\longrightarrow D_{0,A} \longrightarrow \pi_{1*}\cO_{E_{P}}(E_{A}) \longrightarrow 0.$$
Since $c_1\bigl(\pi_{1*}\cO_{E_{P}}(E_{A})\bigr) = c_1\bigl(\cO_{P}(A)\bigr) = 0$, one gets
$c_1(D_{0,A}) = c_1(D_{0,A_0})$. To conclude, we need to show that $c_1(D_{0, P})=0$ for any point $P\in E$.
Indeed we have
$$0 \longrightarrow D_{0,0}\longrightarrow D_{0,P} \longrightarrow \pi_{1*}\cO_{E_{P}}(E_{P}) \longrightarrow R^1 \pi_{1 *} \cO_{E\times E} \longrightarrow 0,$$
which gives that $c_1(D_{0,P})=c_1(R^1 \pi_{1 *} \cO_{E\times E} )$. By Grothendieck-Verdier duality $R^1 \pi_{1 *} \cO_{E\times E}\cong \cO_E$ and the result follows.
\end{proof}

\begin{prop} \label{prop:upperbound4}
Equation (\ref{main_equation}) has at most $16$ distinct solutions in degree $4$.
\end{prop}
\begin{proof}
By Lemma \ref{lem:c1Dma} we have
\begin{align*}
c_1(\cG)&=c_1(D_{3,1})=3,\\
c_1(\cF)&=c_1(D_{2,1})=2,\\
c_1(\cU)&=c_1(D_{4,2})=8.
\end{align*}
Let us denote the quotient $\cU / \cG$ by $\cV$, we have $c_1(\cV)=c_1(\cU)-c_1(\cG)=5$. By following Grothendieck's notation of \cite{H}, we consider the projective bundle $\bP:= \bP\bigl(\cF ^\vee\bigr) \xrightarrow{q} E$ of $\cF$. Note that $\mbox{dim}(\bP)=3$. Denote the class of the line bundle $\cO_{\bP}(1)$ by $\epsilon$.
The map $f\to f^2\omega dz \ {\rm mod} \ \delta(\cG)$ can be viewed globally as a morphism of vector bundles on $\bP$
$$\phi: \cO_\bP(-2)\to q^\ast \cV,$$ that is,  as an element of $H^0\bigl(\bP,q^\ast \cV(2)\bigr)$.
Recalling that $\epsilon^3-q^\ast c_1(\cF^\vee)\epsilon^2=0$, we find $\epsilon^3=q^\ast c_1(\cF^\vee) \epsilon^2 = - q^\ast c_1(\cF) \epsilon^2 = -c_1(\cF) =-2$.
If we denote by $A,B$ and $C$ the Chern roots of the rank $3$ vector bundles $\cV$, we  use the splitting principle to compute:
 \begin{align*}
 &c_3(q^\ast \cV(2))=(A+2\epsilon)(B+2\epsilon)(C+2\epsilon)\\
 &=ABC+2\epsilon(AB+BC+CA)+4\epsilon^2(A+B+C)+8\epsilon^3\\
 &=4\epsilon^2c_1(\cV)+8\epsilon^3=4\bigl(c_1(\cV)+2\epsilon^3\bigr)=4(5-2\times 2)=4.
\end{align*}
This happens for each of the $4$ spin structures on $E$, thus the equation (\ref{main_equation}) has at most $16$ solutions.
\end{proof}

Now we prove the existence of exactly $16$ solutions of equation (\ref{main_equation}) for a particular elliptic curve.
The argument is independent of a fixed theta characteristic $\vartheta$ on $E$. Let $g$  be a solution of equation (\ref{main_equation}). Remember that we defined  $\sigma: E\to E$ to be  the involution that fixes the origin $P$ of the elliptic curve. The function $g^\sigma:=g\circ \sigma$ is then another odd degree $4$ cover triply ramified at $P$. Recall that $Q,R$ and $S$ denote the non-trivial $2$-torsion points on $E$.

\begin{lem}
The solutions $g$ and $g^\sigma$ are different.
\end{lem}
\begin{proof}
Assume $g^\sigma =\pm g$. The unique point in $g^{-1}\bigl(g(P)\bigr)\setminus \{P\}$ is then fixed by $\sigma$ and we may assume this point to be $Q$.
Moreover, $\sigma $ acts on the other triple ramification points of $g$, which we denote by $x, y$ and $z$.
Since $\sigma$ is an involution, there must be at least one fixed point and we may assume $x=R$.
Consider $x'$ to be the remaining point in the fibre of $R$, that is $g^{-1}(g(R))=3R+x'$.
Then $x'$ is also fixed by $\sigma$ and so $x'=S$. Summarizing, we have $g^*(\infty) =3P+Q$ and $g^*(0)=3R+S$.

\vskip 3pt

Let $v_1$, $v_2$ and $v_3$ in $\bC$ be the half periods of $E=\bC/\Lambda$ and set $e_i:=\wp(v_i)$. We consider the equation (\ref{eqn:propP1}) and (\ref{eqn:propP2}) from the preliminaries. In particular, $e_1+e_2+e_3=0$. Recall that $\wp$ has a pole of order $2$ at $0$ and $\wp'$ has a pole of order $3$ at $0$ and on the $v_i$. Consider the function $G:\bC \to \bC$ defined by
$$G(z):= \wp'(z) \frac{\wp(z)-e_2}{\wp(z)-e_1}.$$
The function $G(z)$ has a pole of order $3$ at the points $0$ and $v_1$, and a zero of order $3$ at the point $v_2$. The half period $v_1$ corresponds to the point $Q$ and, similarly, $v_2$ corresponds to $R$. We have that $G(-z)=-G(z)$.
The attached meromorphic function $g_0$ on $E$ satisfies
$$\Div(g_0)=3P+Q-3R-S.$$
Hence up to a constant, $g_0=g$. In particular, $G$ has to be an odd function. To impose this we compute the derivative of $G$, then we study the vanishing locus of its discriminant. To simplify calculations, we set
$$\varphi(z):=\frac{\wp(z)-e_2}{\wp(z)-e_1},$$
in such a way that $G(z)=\varphi(z)\wp'(z)$. We have
$$\varphi'(z)= \wp'(z) \frac{e_2-e_1}{(\wp(z)-e_1)^2}.$$
We proceed with the computation of $G'(z)$ by using the properties of $\wp$ given in equation (\ref{eqn:propP1}):
\begin{equation*}
\begin{aligned}
 G'(z)=& \wp''(z) \varphi(z)+\wp'(z) \varphi'(z)=\wp''(z)\varphi(z)+\wp'(z)^2\frac{e_2-e_1}{(\wp(z)-e_1)^2}\\
=&\wp''(z)\varphi(z)+4(\wp(z)-e_1)(\wp(z)-e_2)(\wp(z)-e_3)\frac{e_2-e_1}{(\wp(z)-e_1)^2}\\
=&\wp''(z)\varphi(z)+\varphi(z)\left(4(\wp(z)-e_3)(e_2-e_1)\right)=\varphi(z)\bigl(\wp''(z)+4(e_2-e_1) (\wp(z)-e_3)\bigr).
\end{aligned}
\end{equation*}
To understand the ramification of $G$, we study the zeroes of
$$\wp''(z)+4(e_2-e_1)(\wp(z)-e_3)=6\wp(z)^2-\frac{1}{2}g_2+ 4(e_2-e_1) (\wp(z)-e_3).$$
Set $v:=\wp(z)$, and recall that from (\ref{eqn:propP2}) that
\begin{equation} \label{eqn:EG}
g_2=-4\left(e_1e_2+e_3(e_1+e_2)\right)=4(e_1^2+e_1e_2+e_2^2).
\end{equation}
By using (\ref{eqn:EG}) and that $e_3=-e_1-e_2$, we finally get the the equation
$$2\left(3v^2+2(e_2-e_1)v+(-3e_1^2-e_1e_2+e_2^2)\right)=0,$$
which has discriminant $\Delta (\tau)=16 \Delta_0(\tau)$, where $\Delta_0(\tau)=10 e_1^2-2e_2^2+e_1 e_2$.
Now we focus on the elliptic curve corresponding to $\tau =i$. For this curve one has $g_3=e_1 e_2 e_3=0$ and $g_2\ne 0$. Therefore, one and only one of the values $e_i$ is zero. If $e_1=0$, then $\Delta_0(i)=-2e_2^2\ne 0$. If $e_2=0$, then $\Delta_0(i)= 10e_1^2\ne 0$. Finally if $e_3=0$, then $e_2=-e_1$ and $\Delta_0(i)= 8e_1^2\ne 0$. We obtain that $g\ne g^\sigma$ as desired.
\end{proof}

Observe that the elliptic curve $E$ we are considering also has an automorphism $j$ such that $j^2=\sigma $. Then $g^j$ is a new function with odd ramification. Moreover, $g^j=g$ would imply
$$g^\sigma =g^{j\circ j}=g^j\circ j=g\circ j=g^j=g,$$
a contradiction. With a similar argument one can prove that $g^j\ne g^\sigma$ and that all the solutions $g$, $g^j$, $g^\sigma$, $g^{j\sigma}$ are different.

\begin{prop} \label{prop:lowerbound4}
Equation (\ref{main_equation}) has at least $16$ distinct solutions in degree $4$.
\end{prop}
\begin{proof}
For two of the possible theta characteristics, namely the trivial and one of the even ones, we can assume that $j^*(\vartheta)\cong \vartheta$. It follows there are exactly $4$ solutions in each case: $g, g^j, g^{\sigma }=g^{j^2}$ and  $g^{j\sigma}=g^{j^3}$.
To show that this is so also for the remaining theta characteristics, we use a monodromy argument. Let us consider the $1$-dimensional spin moduli space
$$\cS_{1,1}^{+}=\left\{[E,p,\vartheta]:[E,p]\in \cM_{1,1},   \vartheta^2\cong\cO_E,\  \vartheta \ncong \cO_E\right\}$$
which is known to be connected with a forgetful map $\cS_{1,1}^+ \to \cM_{1,1}$ of degree $3$.  We have shown that there are $4$ odd meromorphic functions corresponding to a general $[E,p,\vartheta]\in \cS_{1,1}^+$. It follows that for a generic elliptic curve there are $12$ solutions attached to even theta characteristics. The conclusion is that we can find at least $16$ solutions of Equation (\ref{main_equation}) for a generic elliptic curve.
\end{proof}

\subsection{Odd covers of degree $5$ on an elliptic curve.} We establish the following result:

\begin{thm} \label{main_section5}
The number $N_5$ of odd maps of degree $5$ computed in the case of a general elliptic curve is equal to $16$.
\end{thm}

\begin{proof}
Thanks to Proposition \ref{equation}, this problem is equivalent to finding the number of solutions of equation (\ref{main_equation}). The result is proved by combining the upper bound provided by Proposition \ref{prop:upperbound5} and the lower bound provided by Proposition \ref{prop:lowerbound5}.
\end{proof}

This situation is simpler than the one considered in Theorem \ref{main_section4}, since one of the fibres of the map $f:E\rightarrow \bP^1$ is $5P$ and there is no freedom for a new pole. Fix $P$ to be the origin of the curve $E$. We count the solutions of equation (\ref{main_equation}) for a given spin structure $\vartheta$. Assume $f$ has a pole of order $5$ at $P$ and a zero of order $3$ at $Q$, as well as two other ramification points of index $3$. By looking at the local expression of $f$ at $P$ and taking derivatives, we obtain that $df$ has a pole of order $6$ at $P$. Derivation induces a map
$$\delta :  H^0\bigl(E,\cO_E(5P)\bigr) \to  H^0\bigl(E,\omega_E(6P)\bigr)\cong H^0\bigl(E,\cO_E(6P)\bigr).$$
Since the kernel is formed by the constants the image of $\delta$ is $4$-dimensional.
On the one hand, when $\vartheta\cong \cO_E$, we have to consider the map:
$$\alpha:  H^0\bigl(E, \cO_E(3P)\bigr) \to H^0\bigl(E, \cO_E(6P)\bigr), \ \ \mbox{ }   \alpha(s):=s^2dz.$$ On the other hand, when $\vartheta$ is even, the map $\alpha$ has to be defined by
$$\alpha: H^0\bigl(E,\cO_E(2P+Q)\bigr) \to H^0\bigl(E, \cO_E(6P)\bigr),  \ \mbox{ }  \ s\mapsto s^2hdz,$$ where $\Div(h)=2P-2Q$.

\begin{prop} \label{prop:upperbound5}
Equation (\ref{main_equation}) has at most $16$ distinct solutions in degree $5$.
\end{prop}
\begin{proof}
The result of Corollary \ref{cor:intersection} still holds: the solutions of Equation (\ref{main_equation}) lies in the intersection of the images of the maps $\delta$ and $\alpha$, up to constants. Hence, we have to look at the kernel of the map
$$ H^0\bigl(E, \cO_E(3P)\bigr) \to H^0\bigl(E, \cO_E(6P)\bigr) \to \frac {H^0\bigl(E, \cO_E(6P)\bigr)}{\delta \bigl(H^0(E,\cO_E(5P))\bigr)}.$$
By projectivizing, this amounts to considering inside $\bP\bigl(H^0(E,\cO_E(6P))\bigr)\cong \bP^5$  the intersection of the $3$-plane $\bP\bigl(\mbox{Im}(\delta)\bigr)$ with the image
$\bP\bigl(H^0\left(E,\cO_E(3P)\right)\bigr)$ of $\alpha$, which is a Veronese surface. Counting with multiplicities, there are $4$ solutions when $\vartheta$ is trivial. The same argument works also when $\vartheta\ncong \cO_E$. Putting everything together, we obtain $N_5\leq 16$.
\end{proof}

\vskip 3pt

Now we prove the existence of exactly $16$ distinct solutions of equation (\ref{main_equation}) for a particular elliptic curve. We follow the same strategy as in the previous section.
We fix a meromorphic function $g$ inducing an odd map of degree $5$ with a pole at $P$ of order $5$ and three additional triple ramification points $x, y$ and $z$.
Let $\sigma $ be the automorphism of $E$ fixing $P$. Then, $g^\sigma $ is another meromorphic function on $E$ with the same properties. Assume $g^\sigma =\pm g$. Then one of the ramification points must be one of the points $Q, R, S$. We may assume this point to be $x=Q$ and that $g(Q)=0$.

\begin{lem}\label{lem410} Let $g$ be a meromorphic function as described above with $g^\sigma = \pm g$ and $g(Q)=0$. Then $g^*(0)=3Q+R+S$.
\end{lem}
\begin{proof}
Assume $u\in E$ is a zero of $g$ which is not fixed by $\sigma$. Then $\sigma(u)$ must be another zero of $g$. Therefore we get $\Div(g)=3Q+u+\sigma(u)-5P.$
Let now consider $\widetilde G(z):= \wp '(z)(\wp (z)-e_1)$; it is easy to check that the divisor of the meromorphic function $\widetilde g$ induced on $E$ by $\widetilde G$ is $\Div(\widetilde g)=3Q+R+S-5P$.
Then, the divisor of $\frac{g}{\widetilde g}$ is $u+\sigma(u)- R-S$. Hence $u+\sigma (u) \sim R+S$. Since $u+\sigma(u)\sim 2R$, we obtain that $R$ and $S$ are linearly equivalent, which is impossible.
\end{proof}

From Lemma \ref{lem410}, the functions $g$ and $\widetilde g$ have the same attached divisor,  hence they only differ by a constant and we may assume that $g=\widetilde g$.

\begin{prop} \label{prop:lowerbound5}
Equation (\ref{main_equation}) has at least $16$ solutions in degree $5$.
\end{prop}
\begin{proof}
Now we impose the existence of other odd ramification points for $\widetilde g$. We can use the explicit expression of $\widetilde G$ and the properties of the derivatives of $\wp$. It is easy to check that
$$\widetilde G'(z)= (\wp (z)-e_1)\left(6\wp (z)^2-2 (e_1^2+e_1e_2+e_2^2)+4(\wp(z)-e_2)(\wp(z)-e_3)\right).$$
The discriminant of the quadratic part is
$$\Delta(\tau)=16(5e_1^2+6e_2^2+e_3^2+5e_1e_2-8e_2e_3).$$
As in the computation of $N_4$ it turns out that $\Delta(i)\ne 0$ and then $g^\sigma \ne g$. Moreover, we can use the involution $j$ with $j^2=\sigma$ and then there are $4$ different meromorphic functions with the prescribed ramification and the same theta characteristic $\vartheta$ with $j^*(\vartheta)\cong \vartheta$. Then we can apply the same monodromy argument we used in the previous section to finish the proof.
\end{proof}

\section{Counting odd covers of elliptic curves II: an approach via degeneration}\label{sec:deg}

In this section we present a second proof of Theorems \ref{main_section4} and \ref{main_section5}. The proof relies on counting the number of odd admissible covers from a curve stably equivalent to a rational nodal curve, that is, an elliptic curve with $j$-invariant $\infty$.

\vskip 3pt

\subsection{Odd degree $5$ covers of elliptic curves.} We denote by $\hu_{1,5}^{\mathrm{ord}}$ the $1$-dimensional Hurwitz space
parametrizing odd admissible covers $\bigl[f:X\rightarrow \Gamma, \ P, x, y, z \bigr]$, where $X$ (respectively $\Gamma$) is a connected nodal curve of arithmetic genus one (respectively zero), $f$ is a finite map of degree $5$ which is totally ramified at the point $P\in X$ and triply ramified at the mutually distinct points $x,y,z\in X_{\mathrm{reg}}\setminus \{P\}$. The symmetric group $\mathfrak{S}_3$ acts on $\hu_{1,5}^{\mathrm{ord}}$ by permuting the ramification points $x,y$ and $z$ and we denote the quotient by
$$\hu_{1,5}:=\hu_{1,5}^{\mathrm{ord}}/\mathfrak{S}_3.$$ Let $\sigma_5:\hu_{1,5}\rightarrow \mm_{1,1}$ be the map associating to a cover $\bigl[f:X\rightarrow \Gamma,\ P, x+y+z\bigr]$ the stabilization of the source curve, that is, $[\mathrm{st}(X),P]\in \mm_{1,1}$. We shall determine the degree of the generically finite morphism $\sigma_5$.

\vskip 3pt

We denote by $[R,P, U, V]\in \mm_{0,3}$ a fixed $3$-pointed smooth rational curve and set $$[E_{\infty},P]:=\bigl[R/U\sim V, P]\in \mm_{1,1}$$ to be the pointed elliptic curve with $j$-invariant $\infty$. In what follows we explicitly describe the cycle $\sigma_5^*\bigl([E_{\infty},P]\bigr)$.
We shall count (with appropriate multiplicities) the admissible covers in $\hu_5$ having as source a nodal curve stably equivalent to $[E_{\infty},P]$.

\begin{thm}\label{N5:attempt2}
We have that \ $\mathrm{length }\ \sigma_5^*\bigl([E_{\infty}, P]\bigr)=\mathrm{deg} (\sigma_5)=16$. It follows once more that $N_5=16$.
\end{thm}
\begin{proof}
Let $\xi:=\bigl[f:X\stackrel{5:1}\rightarrow \Gamma, \ P, x,
y, z\bigr]\in \hu_{1,5}^{\mathrm{ord}}$ be an admissible cover such that the stabilization of $[X,P]$ is $[E_{\infty},P]$. In particular
$f^{-1}(\Gamma_{\mathrm{sing}})=X_{\mathrm{sing}}$, which implies that $R$ appears as an irreducible component of $X$ and $f(U)=f(V)=:B\in \Gamma_{\mathrm{sing}}$. Indeed, if $f(U)\neq f(V)$, then necessarily $f^{-1} f(R)$ contains another (rational) component of $X$ different from $R$, which is impossible for if we denote $f_R:=f_{|R}$, then $\mbox{deg}(f_{R})=\mbox{deg}(f)=5$ because $f_R$ is fully ramified at $P$. We denote by $R_1$ the subcurve of $X$ meeting $R$ at the points $U$ and $V$. Since the arithmetic genus of $X$ is equal to one, it follows that $R\cap R_1=\{U,V\}$.  We set $f(R)=:\bP^1$ and $f(R_1)=:\bP^1_1$, thus $\bP^1\cap \bP^1_1=\{B\}$.

\vskip 3pt

We claim that the degree of the restriction $f_{1}:=f_{|R_1}$ is at most $4$ and that precisely two of the ramification points $x,y,z$ lie on $R_1$, whereas the remaining point lies on $R$. Indeed, else, regarding $B=f(U)=f(V)$ as a smooth point of $\bP^1=f(R)$, we have $f_R^*(B)=\alpha\cdot U+(5-\alpha)\cdot V$, where $1\leq \alpha\leq 4$. Applying the Hurwitz formula to  $f_R:R\rightarrow \bP^1$, we obtain
that apart from $U, V$ and $P$, the cover $f_R$ has precisely one ramification point contributing with multiplicity one to the ramification divisor of $f_R$, which is impossible, for  $f$ is an odd map. It follows that the fibre $f_R^{-1}(B)$ contains a third point $U'\in R\setminus\{U,V\}$.

\vskip 3pt

We now turn our attention to the cover $f_1:R_1\rightarrow \bP^1_1$. We have seen that $\mbox{deg}(f_1)\leq 4$. If the degree of $f_1$ is two or three, then the Hurwitz formula implies that $f_1$ has a \emph{simple} ramification point in $R_1\setminus \{U,V\}$, which is impossible. Thus $\mbox{deg}(f_1)=4$ and we write $f_1^*(B)=\alpha\cdot U+(4-\alpha)\cdot V$, where $1\leq \alpha\leq 3$. It follows that the map $f_R$ is unramified at $U'$ and that
$$f_R^*(B)=\alpha\cdot U+(4-\alpha)\cdot V+U'\in \mbox{Div}(R).$$
Assume $x\in R\setminus\{U,V, P\}$ is the triple ramification point of $f$ lying on $R$, whereas $\{y,z\}\in R_1\setminus\{U,V\}$ are
the triple ramification point of $f$ lying on $R_1$. Note that the curve $X$ consists of a further rational component $R_2$ mapping isomorphically onto $f(R_1)$ and meeting $R$ at the point $U'$.

\vskip 3pt

We distinguish two cases depending on whether $\alpha\in \{1,3\}$, or $\alpha=2$.

\vskip 3pt

\noindent {\bf (i)} Assume $\alpha=1$, that is, $f_R^*(B)=U+3\cdot V+U'$
and $f_1^*(B)=U+3\cdot V\in \mbox{Div}(R_1)$. We claim that up to the $PGL(2)$-action on the base, there exists a \emph{unique} such cover $f_R:R\rightarrow \bP^1$. We may indeed assume
$P=\infty \in R$, $U=1$ and $V=0\in R$. Then the function
$$f_R(t)=t^3(t-1)(t-b),$$
where $b\in \mathbb C$ has a pole of order five at $P$. Imposing the condition that $f$ have triple ramification at a further point $x\in R\setminus\{P,U,V\}$, we obtain that $4b^2-7b+4=0$, thus there are \emph{two} choices for $f_R$. Observe that these two covers lead to genuinely different points in $\hu_{1,5}$ (in particular also in $\hu_{1,5}^{\mathrm{ord}}$), for $f_R$ has no non-trivial automorphisms. Indeed such an automorphism $\tau_R \in \mbox{Aut}(R)$  fixes the point $P$ of total ramification, the unique triple ramification point $x$ of $f_R$, as well as the point $U'\in f^{-1}(B)\setminus \{U,V\}$. Thus $\tau_R=\mbox{Id}_R$.

\vskip 4pt

We now consider the $R_1$-side and assume $U=1, V=\infty\in R_1$ and $y=0\in R_1$. We may write
$f_1(t)=\frac{t^3(t+a)}{t-1}$. Imposing the condition that $f_1$ have a further triple ramification point $z$, we find
\begin{equation}\label{f1form}
f_1(t)=\frac{t^3(t-4)}{t-1},
\end{equation}
and $f_1^*(0)=3\cdot 0+1\cdot 4\in \mbox{Div}(R_1)$ and $f_1^*(-16)=3\cdot 2+1\cdot (-2)\in \mbox{Div}(R_1)$. In particular, $z=2$ is also a triple ramification point of $f_1$.

\vskip 3pt

Observe now that $f_1$ has an automorphism $\tau_1$ of order $2$ that fixes the points $U$ and $V$ and interchanges the ramification points
$y$ and $z$. Using (\ref{f1form}), we find $\tau_1(t)=2-t$. This implies that the map $\hu_{1,5}^{\mathrm{ord}}\rightarrow \hu_{1,5}$ is ramified with order $2$ at such a point $\xi=\bigl[f:X\rightarrow \Gamma, P, x,y,z\bigr]\in \hu_{1,5}^{\mathrm{ord}}$.

\vskip 3pt

The following local statement is essential in the proof of both Theorems \ref{N5:attempt2} and \ref{N4:attempt2}.

\vskip 4pt

\noindent {\bf Claim:} The map $\hu_{1,5}^{\mathrm{ord}}\rightarrow \mm_{1,1}$ is ramified with order $4=\mbox{ord}_U(f)+\mbox{ord}_V(f)$ at the point $\xi\in \hu_{1,5}^{\mathrm{ord}}$.

\vskip 4pt

Assuming this fact for a moment, we conclude that the contribution to the cycle $\sigma_5^{*} \bigl([E_{\infty},P]\bigr)$ coming from the case $\alpha=1$ is equal to $4=2\times 4\times \frac{1}{2}=4$: We multiply by $2$ for the two choices of $f_R$, by $4$ because of the ramification of the map $\hu_{1,5}^{\mathrm{ord}}\rightarrow \mm_{1,1}$ at each of the points $\xi$ and divide by $2$ because of the existence of the automorphism of $f$ which is trivial along $R$ and $R_2$, while being equal to $\tau_1$ along $R_1$. The case $\alpha=3$ is identical (one switches the role of $u$ and $v$). Summarizing the discussion so far, we have identified a subcycle of length $8=4+4$ of
$\sigma_5^* \bigl([E_{\infty},P]\bigr)$ coming from the cases $\alpha \in \{1,3\}$.

\vskip 4pt

\noindent {\emph{Proof of the claim.}} We show that $\hu_{1,5}^{\mathrm{ord}}\rightarrow \mm_{1,1}$ is ramified with order $4$ at the point $\xi$. Let
\begin{equation}\label{univfamily}
F:\mathcal{X}\rightarrow \mathcal{P}
\end{equation}
be the universal degree $5$ admissible cover over $\hu_{1,5}^{\mathrm{ord}}$. One has a finite map $\mathfrak{b}:\hu_{1,5}^{\mathrm{ord}}\rightarrow \mm_{0,4}$ associating to an admissible cover $[f:X\rightarrow \Gamma, P, x, y, z]$ the point $[\Gamma, f(P), f(x), f(y), f(z)]\in \mm_{0,4}$. According to the local description (\ref{localring}) of the local ring of $\hu_{1,5}^{\mathrm{ord}}$ at
the point $\xi$, we have that
$$\hat{\cO}_{\xi, \hu_{1,5}^{\mathrm{ord}}}\cong \mathbb C\bigl[\bigl[s_1,s_2\bigr]\bigr]/s_1^3=s_2=t,$$ where $t$ is the local parameter on $\mm_{0,4}$ corresponding to the boundary point $\mathfrak{b}(\xi)$. Around the points $(\xi, V)$ and $(\xi, U)\in \mathcal{X}$, the cover $F$ considered in (\ref{univfamily}) has the following local expression:

$$\mathcal{X} \ \mbox{ around } (\xi, V): v_1v_2=s_1, \ \mbox{  } \  \mathcal P \mbox{ around } F(\xi,V): \gamma_1\gamma_2=s_1^3,
\ \mbox{ the map } F:\gamma_1=v_1^3, \gamma_2=v_2^3,$$ respectively
$$\mathcal{X} \ \mbox{ around } (\xi, U): u_1u_2=s_1, \ \mbox{  } \  \ \mathcal P \mbox{ around } F(\xi,U): \gamma_1\gamma_2=s_1,
\ \mbox{ the map } F:\gamma_1=u_1, \gamma_2=u_2.$$
We consider the map $\mbox{Spec } \mathbb C[[t]]\rightarrow \hat{\mathcal{O}}_{\xi, \hu_{1,5}^{\mathrm{ord}}}$ given by sending $t\mapsto (s_1=t, s_2=t^3)$.
The induced family of curves $\mathcal{X}\times_{\hu_{1,5}^{\mathrm{ord}}} \mbox{Spec } \mathbb C[[t]]\rightarrow \mbox{Spec } \mathbb C[[t]]$ has local equation
$v_1v_2=t$ around the point $(\xi,V)$, and $u_1u_2=t^3$ around the point $(\xi,U)$ respectively. The fibre over $0$ consists of the nodal genus one curve $X=R\cup R_1\cup R_2$. We first blow-down the $(-1)$-curve $R_2$ and then $R_1$. The resulting curve
$\mathcal{X'}\rightarrow \mbox{Spec } \mathbb C[[t]]$ is the family of curves induced by base-change from $F$ under the map $\hu_{1,5}^{\mathrm{ord}}\rightarrow \mm_{1,1}$. Its central fibre is $[E_{\infty}, P]$ and the local equation of $\mathcal{X'}$ around the unique node of the central fibre is
$$v_1v_2=t\cdot t^3=t^4,$$
which finishes the proof of the claim.

\vskip 3pt

We now proceed with the proof of the remaining cases of Theorem \ref{N5:attempt2}.

\vskip 4pt

\noindent {\bf (ii)} Assume now $\alpha=2$, thus $f_R^*(B)=2\cdot U+2\cdot V+U'\in \mbox{Div}(R)$ and $f_1^*(B)=2\cdot U+2\cdot V\in \mbox{Div}(R_1)$. In order to count the number of such maps $f_R$, assume again $P=\infty\in R$, $U=1$ and $V=0\in R$. Writing $f_R(t)=t^2(t-1)^2(t-b)$,
the condition that $f_R$ has a triple ramification point $x\in R\setminus\{0,1,\infty\}$ leads to the equation $16b^2-16b+9=0$, thus to two choices for $f_R$. The same argument as in the case (i) shows that $f_R$ has no automorphism, nor are the found maps equivalent under
the $PGL(2)$-action. We now consider the $R_1$-side and set $U=1, V=0\in R_1$ and $y=\infty\in R_1$. Up to the $PGL(2)$-action on the base $\bP^1_1$ of the map $f_1:R_1\rightarrow \bP^1_1$ we find two solutions, namely
\begin{equation}\label{eq9}
f_1(t)=\frac{48 \sqrt{3} t^2(t-1)^2}{\bigl(-2t+1+\sqrt{3}\bigr) \bigl(\sqrt{3}+6t-3\bigr)^3}
\ \mbox{ and  } \ \tilde{f}_1(t)=\frac{t^2(t-1)^2}{t-\frac{1}{2}-\frac{\sqrt{3}}{4}}.
\end{equation}
Neither $f_1$ nor $\tilde{f}_1$ have non-trivial automorphisms. Denote by $\tau_1:R_1 \rightarrow R_1$ the automorphism fixing $0$ and $1$ and such that
$$\tau_1(\infty)=\frac{1}{2}+\frac{\sqrt{3}}{6} \ \mbox{ and } \tau_1 \bigl(\frac{1}{2}-\frac{\sqrt{3}}{6}\bigr)=\infty.$$
Then $\tilde{f}_1\circ \tau_1=f_1$. Via the automorphism $\tau\in \mbox{Aut}(X)$ such that $\tau_{R}=\mbox{Id}_R$ and $\tau_{|R_1}=\tau_1$, it follows that $f_1$ and $\tilde{f}_1$ lead to the same point of $\hu_{1,5}^{\mathrm{ord}}$. An argument identical to the one in the claim shows that around each such point $\xi$, the map $\hu_{1,5}^{\mathrm{ord}}\rightarrow \mm_{1,1}$ is ramified with order $4$. Summarizing, the contribution to
the cycle $\sigma_5^*\bigl([E_{\infty},P]\bigr)$ coming from case (ii) is equal to $2\times 4=8$.

\vskip 4pt

None of the points $\xi\in \hu_{1,5}^{\mathrm{ord}}$ found in this proof carry an automorphism fixing all the branch points, hence they all correspond to smooth points of
$\hu_{1,5}^{\mathrm{ord}}$. Putting cases (i) and (ii) together, we conclude that the degree of the map $\sigma_5$ equals $16=8+8$, which finishes the proof.
\end{proof}

\vskip 4pt

\subsection{Odd degree $4$ covers of elliptic curves.} We denote by $\hu_{1,4}^{\mathrm{ord}}$ the $1$-dimensional Hurwitz space
parametrizing odd admissible covers $\bigl[f:X\rightarrow \Gamma, \ P, x, y, z \bigr]$, where $X$ (respectively $\Gamma$) is a connected nodal curve of arithmetic genus one (respectively zero), $f$ is a finite map of degree $4$ which is triply ramified at the point $P\in X$ and at the pairwise distinct points $x,y,z\in X_{\mathrm{reg}}\setminus \{P\}$. The symmetric group $\mathfrak{S}_3$ acts on $\hu_{1,4}^{\mathrm{ord}}$ by permuting $x,y$ and $z$. Let
$$\hu_{1, 4}:=\hu_{1, 4}^{\mathrm{ord}}/\mathfrak{S}_3$$ be the quotient and let $\sigma_4:\hu_{1, 4}\rightarrow \mm_{1,1}$ be the map associating to a cover $\bigl[f:X\rightarrow \Gamma,\ P, x+y+z\bigr]$ the stabilization of the source curve.

\begin{thm}\label{N4:attempt2}
We have that \ $\mathrm{length }\ \sigma_4^*\bigl([E_{\infty},P]\bigr)=\mathrm{deg} (\sigma_4)=16$. It follows once more that $N_4=16$.
\end{thm}
\begin{proof}
We proceed along the lines of the proof of Theorem \ref{N5:attempt2}, highlighting the things that are different. We start with an admissible cover $\xi:=\bigl[f:X\stackrel{4:1}\rightarrow \Gamma, \ P, x,
y, z\bigr]\in \hu_{1,4}^{\mathrm{ord}}$ such that the stabilization of $[X,P]$ is $[E_{\infty},P]$. As before, $R$ appears as an irreducible component of $X$ and $f(U)=f(V)=:B\in \Gamma_{\mathrm{sing}}$. We denote by $R_1$ the subcurve of $X$ meeting $R$ precisely at the points $U$ and $V$. We set $f_R:=f_{|R}:R\rightarrow \bP^1$ and$f_1:=f_{|R_1}:R_1\rightarrow \bP^1_1$, where $\Gamma=\bP^1 \cup_B \bP^1_1$.

\vskip 3pt

There are \emph{three} types of admissible covers possible for $\xi$. First, we could have $\mbox{deg}(f_R)=\mbox{deg}(f_1)=4$ and $f_R^*(B)=\alpha\cdot U+(4-\alpha)\cdot V\in \mbox{Div}(R)$ and $f_1^*(B)=\alpha\cdot U+(4-\alpha)\cdot V\in \mbox{Div}(R_1)$.
Let $x\in R\setminus\{P, U, V\}$ be the remaining triple ramification point of $f_R$ and we denote by $\{y,z\}\subseteq R_1\setminus\{U,V\}$ the remaining triple ramification points of $f_1$.

\vskip 3pt

\noindent {\bf (i)} $\alpha=1$.   Setting $P=0, U=1, V=\infty\in R$, we find a unique solution for $f_R$,  the one given by (\ref{f1form})
$$f_R(t)=\frac{t^3(t-4)}{t-1}.$$ The cover $f_R$ has no automorphism, for such an automorphism $\tau_R$ would have to fix both the marked
point $P$, as well as $U$ an $V$, hence $\tau_R=\mbox{Id}_R$. On the $R_1$-side, setting $U=1, V=\infty\in R_1$ and $y=0$, we have a unique choice for $f_1$ given by the same formula (\ref{f1form}). However, in this case, as in the proof of Theorem \ref{N5:attempt2}, $f_1$ does have an automorphism $\tau_1\in \mbox{Aut}(R_1)$ which fixes both points $U$ and $V$ and interchanges the ramification points $y$ and $z$. The map
$\hu_{1,4}^{\mathrm{ord}}\rightarrow \mm_{1,1}$ is ramified with order $4$ at the point $\xi$. All in all, one gets a contribution of $4=2\times 4\times \frac{1}{2}$ to the cycle $\sigma_4^*\bigl([E_{\infty}, P]\bigr)$ coming from the case when $\alpha\in \{1,3\}$. The factor $\frac{1}{2}$ is explained by the simple ramification of the map $\hu_{1,4}^{\mathrm{ord}}\rightarrow \hu_{1,4}$ at the point $\xi$.

\vskip 5pt

\noindent {\bf (ii)} $\alpha=2$. Setting $U=0, V=1\in R_1$ and $y =\infty\in R_1$, following (\ref{eq9}) we find two solutions for $f_1$, which are related via the $PGL(2)$-action on $R_1$. On the $R$-side, we set also $U=0, V=1\in R$ and $P=\infty \in R$, by using once more (\ref{eq9}) we find two solutions for $f_R$, which this time are not equivalent to one another, for an automorphism $\tau_R$ has to fix $U, V$, as well as $P$,
hence $\tau_R=\mbox{Id}_R$. All in all, we get a contribution of $2\times 4=8$ to the cycle $\sigma_4^*\bigl([E_{\infty},P]\bigr)$, where the factor $4$ equals the ramification index of the map $\hu_{1,4}^{\mathrm{ord}}\rightarrow \mm_{1,1}$ at each of the points $\xi$ considered.

\vskip 5pt

\noindent {\bf (iii)} This is the situation which has no equivalent in the proof of Theorem \ref{N5:attempt2}. In this case $\mbox{deg}(f_R)=3$ and $\mbox{deg}(f_1)=3$. The components $R$ and $R_1$ meet at the points $U$ and $V$ and
$f^{-1}(B)=\{U,V, U', V'\}$, where $\{U'\}= R\cap R_3$ and $\{V'\}=R_1\cap R_2$. Here $R_2$ and $R_3$ are smooth rational curves mapping isomorphically onto $\bP^1=f(R)$ and $\bP^1_1=f(R_1)$ respectively. Thus $f^{-1}(\bP^1)$ is the disjoint union of $R$ and $R_2$, whereas $f^{-1}(\bP^1_1)$ is the disjoint union of $R_1$ and $R_3$. Note that $f$ is unramified over the node $B$ of $\Gamma$.

\vskip 3pt

Modulo the $PGL(2)$-action, there are two choices for a map $f_R:R\rightarrow \bP^1$ of degree $3$  triply ramified at $P$ and at a further unspecified point $x\in R\setminus \{P\}$ and satisfying $f_R(U)=f_R(V)$. In coordinates, if we set $P=\infty$, $U=1$ and $V=0$, then the two choices are
$$f_R(t)=\Bigl(t-\frac{1}{2}+i\frac{\sqrt{3}}{6}\Bigr)^3 \ \ \mbox{ and } \ \ \ \tilde{f}_R(t)=-\Bigl(t-\frac{1}{2}-i\frac{\sqrt{3}}{6}\Bigr)^3.$$
Observe that if $\tau_R\in \mbox{Aut}(R)$ is the automorphism given by $\tau_R(t)=1-t$, thus $\tau_R(U)=V$, $\tau_R(V)=U$ and $\tau_R(P)=P$, then $\tilde{f}_R=f_R\circ \tau_R$.
The same applies for the component $R_1$ of $X$. There are two ways, say $f_{1}$ and $\tilde{f}_{1}$ of choosing a degree $3$ map triply ramified
at both $y$ and $z$ and having $U$ and $V$ in the same fibre. The maps $f_{1}$ and $\tilde{f}_{1}$ are related by an automorphism $\tau_{1}$ of $R_1$ which interchanges $U$ and $V$ and fixes $y$. We find that in total there are  there are \emph{two} points in $\hu_{1,4}$ of this type in $\sigma_4^{-1}\bigl([E_{\infty}, P]\bigr)$.

\vskip 3pt

%Observe that $f$ has an automorphism of order two $\tau:X\rightarrow X$. It has $\tau_{|R_2}=\mbox{Id}_{R_2}$ and $\tau_{|R_3}=\mbox{Id}_{R_3}$, whereas $\tau_{|R}\in \mbox{Aut}(R)$, $\tau_{|R_1}\in \mbox{Aut}(R_1)$, as well as
%$$\tau(U)=V, \ \tau(V)=U, \ \tau(V')=V', \ \tau(U')=U', \ \tau(P)=x, \  \tau(x)=P, \ \tau(y)=z, \ \tau(z)=y.$$
%This shows that the map $\hu_{1,4}^{\mathrm{ord}}\rightarrow \hu_{1,4}$ is ramified to order $2$ at the point $\xi$.

A similar calculation like in the Claim in the proof of Theorem \ref{N5:attempt2} shows that the map $\hu_{1,4}^{\mathrm{ord}}\rightarrow \mm_{1,1}$ is ramified to order $2$ at both these points. All in all, we have a contribution of $4=2\times 2$ to the cycle $\sigma_4^{-1}\bigl([E_{\infty}, P]\bigr)$ coming from case (iii).

\vskip 3pt

Summarizing cases (i), (ii), (iii), we find  that $\mbox{length } \sigma_4^*\bigl([E_{\infty},P]\bigr)=4+8+4=16$, which finishes the proof.

\end{proof}

\section{The generating series of alternating Catalan numbers}
In this Section we explain how using basic facts from Schubert calculus coupled with the Lagrange Inversion formula one can derive
from Theorem \ref{degeneration_formula} both Theorems \ref{main} and \ref{maingen}.

\vskip 3pt

We fix $V:=\mathbb C^{2g+2}$ and set $\mathbb G:=G(2,V)$.
Recall the notation $\sigma_{\alpha_1,\alpha_0}$ for the Schubert cycle in $\bG$. We write $\sigma_{\alpha}:=\sigma_{\alpha{a}, 0}$ for each $\alpha \geq 1$. It is well-known that
$\sigma_1$ is a hyperplane section of $\bG$ in its Pl\"ucker embedding. In particular, $C_{2g}=\sigma_1^{2g}=\mbox{deg}(\bG)=\frac{1}{2g+1}{4g\choose 2g}$. We also recall Giambelli's formula
$\sigma_{\alpha_1,\alpha_0}=\sigma_{\alpha_1}\cdot \sigma_{\alpha_0}-\sigma_{\alpha_1+1}\cdot \sigma_{\alpha_0-1}\in H^*(\bG, \bZ)$.

\vskip 3pt

It is also known that $H^*(\bG, \mathbb Z)$ is generated by the classes $\sigma_1$ and $\sigma_2$ and the top intersection products involving these two classes are given by the following formula, see e.g. \cite{O} Remark 3.4

\begin{equation}\label{oliveira}
\sigma_1^{2m}\sigma_2^{2g-2m}=\sum_{i=0}^{2g-m} (-1)^i{2g-m \choose i}C_{2g-i}.
\end{equation}

\vskip 3pt

We are now in a position to prove Theorem \ref{main}:

\vskip 3pt

\noindent \emph{Proof of Theorem \ref{main}}.
Since we have shown that $N_4=N_5=16$, Theorem \ref{degeneration_formula} can be rewritten as
$$\mathfrak{A}_g=16^g\bigl(\sigma_4+\sigma_{3,1}\bigr)^g.$$
We are going to rewrite this expression in terms involving only the products appearing in (\ref{oliveira}).
Firstly, Giambelli's formula yields $\sigma_4+\sigma_{3,1}=\sigma_1 \sigma_3$. Applying Giambelli's formula
once more, we obtain $\sigma_3=\sigma_1\sigma_2-\sigma_{2,1}=\sigma_1(\sigma_2-\sigma_{1,1}$, where for the last equality we
have used Pieri's formula. One final application of Giambelli's formula yields $\sigma_{1,1}=\sigma_1^2-\sigma_2$, implying
$\sigma_3=2\sigma_1\sigma_2-\sigma_1^3$. Thus
$$\sigma_1^g\sigma_3^g=\sigma_1^{2g}(2\sigma_2-\sigma_1^2)^g=\sum_{k=0}^g\sum_{i=0}^{g-k}(-1)^{k+i}2^{g-k}{g\choose k}{g-k\choose i}C_{2g-i}$$
$$=\sum_{i=0}^g (-1)^i\Bigl(\sum_{k=0}^{g-k} (-1)^k 2^{g-k}{g\choose k}{g-k\choose i}\Bigr) C_{2g-i}=\sum_{i=0}^g (-1)^i2^i{g\choose i}C_{2g-i},$$
where we have used the identity $\sum_{k=0}^{g-i} (-1)^k2^{g-k}{g\choose k}{g-k\choose i}={g\choose i}2^i.$ This brings the proof to an end.
\hfill $\Box$

\subsection{Lagrange inversion for alternating Catalan numbers.}
In order to determine the generating function of the alternating Catalan numbers we use Lagrange inversion. The help of D. Oprea in this section is gratefully acknowledged.

\vskip 3pt

For a power series $f(w)=\sum_{n\geq 0} a_nw^n\in \mathbb Q[[w]]$, we denote its coefficients by $[w^n]f(w):=a_n$.
Suppose one can find two power series $\psi(z)$ and $\phi(z)$ with $[z^0]\phi(z)\neq 0$, such that the function $f(w)=\sum_{n\geq 0} a_nw^n$ can be written as
\begin{equation}\label{lagrangeform}
f(w)=\sum_{n\geq 0} w^n\Bigl([z^n](\psi(z)\phi^n(z))\Bigr).
\end{equation}
Then there exists a unique power series $u=u(w)$ such that $u(w)=w\phi\bigl(u(w)\bigr)$. Moreover one has
\begin{equation}\label{lagrangeinv}
f(w)=\frac{\psi(u)}{\phi(u)}\cdot \frac{du}{dw}=\frac{\psi(u)}{1-w\phi'(u)},
\end{equation}
where we refer to \cite{GJ} 1.2.4 for further details and examples.

\vskip 4pt

We shall now bring the generating function of the alternating Catalan numbers to the form (\ref{lagrangeform}).
To that end, for any $a\in \mathbb R$ we introduce the symbol
$${a\choose n}:=\frac{a(a-1)\cdots (a-n+1)}{n!},$$
thus we have $(1+z)^a=\sum_{n\geq 0} {a\choose n} z^a$.
With this notation, we observe that the Catalan numbers $C_n$ can be rewritten as
$$C_n=\frac{1}{n+1}{2n\choose n}=(-1)^n 2^{2n+1}{\frac{1}{2}\choose n+1}.$$
Using the expression of the alternating Catalan numbers from Theorem \ref{main}, we then have
$$\mathfrak{A}_g=16^g\sum_{s=0}^g (-2)^s{g\choose s} C_{2g-s}=\sum_{s=0}^g 2^{8g-s+1}{g\choose s}{\frac{1}{2}  \choose 2g-s+1}$$
$$=[z^{2g+1}]\  2^{8g+1}\Bigl(1+\frac{z}{2}\Bigr)^g \ \Bigl(1+z\Bigr)^{\frac{1}{2}}.$$
We can now complete the proof of Theorem \ref{maingen}.

\vskip 3pt

\noindent {\emph{Proof of Theorem \ref{maingen}.}
We introduce the auxiliary functions
$$\phi(z)=16\ \bigl(1+\frac{z}{2}\bigr)^{\frac{1}{2}} \ \ \mbox{ and } \ \ \psi(z)=\frac{1}{8}\ \bigl(1+z\bigr)^{\frac{1}{2}}\ \bigl(1+\frac{z}{2} \bigr)^{-\frac{1}{2}},$$
then form the function $f(w)=\sum_{n\geq 0} w^n\ [z^n]\bigl(\phi^n(z)\psi(z)\bigr)$.
Then the function $h(w):=\frac{1}{2}\bigl(f(w)-f(-w)\bigr)$ retaining only the odd coefficients of $f$ can be rewritten as
$$h(w)=\sum_{g\geq 0} w^{2g+1} [z^{2g+1}] \Bigl(1+\frac{z}{2}\Bigr)^g\ \Bigl(1+\frac{z}{2}\Bigr)^{-\frac{1}{2}}\cdot 2^{4(2g+1)-3}
=\sum_{g\geq 0} \mathfrak{A}_g w^{2g+1},$$
that is, $h(w)$ is the generating function of all alternating Catalan numbers. In order to apply (\ref{lagrangeinv}), we introduce the function $u=u(w)$ such that
$u=w\phi(u)$, from which we find
\begin{equation}\label{changevar}
w=\frac{u}{16\sqrt{1+\frac{u}{2}}}, \ \ \mbox{ or equivalently } \ \ \ u=16\bigl(4w+\sqrt{16w^2+1} \bigr).
\end{equation}
We compute $\frac{dw}{du}=\frac{u+4}{32(u+2)\sqrt{1+\frac{u}{2}}}$, hence the Lagrange inversion formula (\ref{lagrangeinv}) leads to
$$f(w)=\frac{\sqrt{\bigl(1+u\bigr)\ \bigl(1+\frac{u}{2}\bigr)}}{2(u+4)},$$
or equivalently
\begin{equation}\label{fmod}
f(w)=\frac{\sqrt{64w^2+1+16w\sqrt{16w^2+1}}}{8\sqrt{16w^2+1}},
\end{equation}
which leads to the claimed formula for $h(w)=\frac{1}{2}\bigl(f(w)-f(-w)\bigr)$.
\hfill $\Box$

\end{document}